\documentclass[11pt]{article}
\usepackage{graphicx}
\usepackage{amsthm, amsmath, amssymb, tikz, bbm}
\usepackage{mathtools}
\usepackage{xifthen}
\usepackage[colorlinks=true,
linkcolor=blue,citecolor=blue,
urlcolor=blue]{hyperref}

\usepackage[margin=1in]{geometry} 

\usepackage[shortlabels]{enumitem}
\usepackage{todonotes}
\usetikzlibrary{calc,shapes, backgrounds}

\newtheorem{lemma}{Lemma}
\newtheorem{theorem}{Theorem}
\newtheorem{corollary}{Corollary}

\newtheorem{remark}{Remark}
\newtheorem{conjecture}{Conjecture}
\newtheorem{proposition}{Proposition}

\newcommand{\dss}{\displaystyle\sum}

\newcommand{\lp}{\left (}
\newcommand{\rp}{\right )}

\newcommand{\cG}{\mathcal{G}}

\newcommand{\cI}{\mathcal{I}}

\newcommand{\CR}{{\rm cr}}

\begin{document}
\title{Anti-Ramsey number of edge-disjoint rainbow spanning trees}

\author{
Linyuan Lu
\thanks{University of South Carolina, Columbia, SC 29208,
({\tt lu@math.sc.edu}). This author was supported in part by NSF
grant DMS-1600811.} \and
Zhiyu Wang \thanks{University of South Carolina, Columbia, SC 29208,
({\tt zhiyuw@math.sc.edu}).} 
}

\maketitle
\begin{abstract}
An edge-colored graph $G$ is called \textit{rainbow} if every edge of $G$ receives a different color. The \textit{anti-Ramsey} number of $t$ edge-disjoint rainbow spanning trees, denoted by $r(n,t)$, is defined as the maximum number of colors in an edge-coloring of $K_n$ containing no $t$ edge-disjoint rainbow spanning trees. Jahanbekam and West [{\em J. Graph Theory, 2016}] conjectured that for any fixed $t$, $r(n,t)=\binom{n-2}{2}+t$ whenever $n\geq 2t+2 \geq 6$. In this paper, we prove this conjecture. We also determine $r(n,t)$ when $n = 2t+1$. Together with previous results, this gives the anti-Ramsey number of $t$ edge-disjoint rainbow spanning trees for all values of $n$ and $t$.
\end{abstract}

\section{Introduction}

An edge-colored graph $G$ is called \textit{rainbow} if every edge of $G$ receives a different color.
The general \textit{anti-Ramsey problem} asks for the maximum number of colors $AR(n,\cG)$ in an edge-coloring of $K_n$ containing no rainbow copy of any graph in a class $\cG$. For some earlier results when $\cG$ consists of a single graph, see the survey \cite{FMO}. In particular, Montellano-Baallesteros and Neumann-Lara \cite{MN} showed a conjecture of Erd\H{o}s, Simonovits and S\'os \cite{ESS} by computing $AR(n, C_k)$. Jiang and West \cite{JW} determined the anti-Ramsey number of the family of trees with $m$ edges. 

Anti-Ramsey problems have also been investigated for rainbow spanning subgraphs. In particular, Hass and Young \cite{HY} showed that the anti-Ramsey number for perfect matchings (when $n$ is even) is $\binom{n-3}{2}+2$ for $n\geq 14$. For spanning trees, Bialostocki and Voxman \cite{BV} showed that the maximum number of colors in an edge-coloring of $K_n$ $(n\geq 4)$ with no rainbow spanning tree is $\binom{n-2}{2}+1$. Jahanbekam and West \cite{West} extended the investigations to finding the anti-Ramsey number of $t$ edge-disjoint rainbow spanning subgraphs of certain types including matchings, cycles and trees. In particular, for rainbow spanning trees, let $r(n,t)$ be the maximum number of colors in an edge-coloring of $K_n$ not having $t$ edge-disjoint rainbow spanning trees. Akbari and Alipour \cite{AA} showed that $r(n,2) = \binom{n-2}{2}+2$ for $n\geq 6$. Jahanbekam and West \cite{West} showed that
	 $$r(n,t) = \begin{cases} 
	                     {n-2\choose 2}+t   & \textrm{ for } n >  2t+\sqrt{6t-\frac{23}{4}}+\frac{5}{2}\\
     			{n\choose 2}-t & \textrm{ for } n = 2t,
 		      \end{cases}$$
and they made the following conjecture:

\begin{conjecture}{\cite{West}}
$r(n,t) = \binom{n-2}{2} + t$ whenever $n\geq 2t+2 \geq 6.$
\end{conjecture}

In this paper, we show that the above conjecture holds and we also determine the value of $r(n,t)$ when $n = 2t+1$.  Together with previous results (\cite{BV},\cite{AA},\cite{West}), this gives the anti-Ramsey number of $t$ edge-disjoint rainbow spanning trees for all values of $n$ and $t$. 
\begin{theorem}\label{anti-Ramsey}
	For all positive integers $t$,
 $$r(n,t) = \begin{cases} 
	                     {n-2\choose 2}+t   & \textrm{ for } n \geq 2t+2\\
			{n-1\choose 2}   & \textrm{ for } n = 2t+1\\
     			{n\choose 2}-t & \textrm{ for } n = 2t,
 		      \end{cases}$$
\end{theorem}

\begin{remark}
	Note that if $n < 2t$, then $K_n$ does not have enough edges for $t$ edge-disjoint spanning trees. 
\end{remark}

The main tools we use are two structure theorems that characterize the existence of $t$ color-disjoint rainbow spanning trees or the existence of a \textit{color-disjoint} extension of $t$ edge-disjoint rainbow spanning forests into $t$ edge-disjoint rainbow spanning trees. When $t=1$, Broersma and Li \cite{BL} showed that determining the largest rainbow spanning forest of a graph can be solved by applying the Matroid Intersection Theorem. The following characterization was established by Schrijver \cite{Schrijver} using matroid methods, and later given graph theoretical proofs by Suzuki \cite{Suzuki} and also by Carraher and Hartke \cite{Carraher-Hartke}.

\begin{theorem}{(\cite{Schrijver, Suzuki, Carraher-Hartke})}\label{Suzuki}
An edge-colored connected graph $G$ has a rainbow spanning tree if and only if for every $2\leq k\leq n$ and every partition of $G$ with $k$ parts, at least $k-1$ different colors are represented in edges between partition classes.
\end{theorem}

The above results can be generalized to $t$ \textit{color-disjoint} rainbow spanning trees using similar matroid methods by Schrijver \cite{Schrijver}. For the sake of self-completeness, we reproduce the proof using matroid methods in Section \ref{sec:color-disjoint}. We also give a new graph theoretical proof of Theorem \ref{partition}.

\begin{theorem}\cite{Schrijver}\label{partition}
An edge-colored multigraph $G$ has $t$ pairwise color-disjoint rainbow spanning trees if and only if for every partition $P$ of $V(G)$ into $|P|$ parts, at least $t(|P|-1)$ distinct colors are represented in edges between partition classes.
\end{theorem}

\begin{remark}
Recall the famous Nash-Williams-Tutte Theorem (\cite{Nash, Tutte}): A multigraph contains $t$ edge-disjoint spanning trees if and only if for every partition $P$ of its vertex set, it has at least $t(|P|-1)$ cross-edges. Theorem \ref{partition} implies the Nash-Williams-Tutte Theorem by assigning every edge of the multigraph a distinct color.
\end{remark}

Theorem \ref{partition} can be also generalized to extend edge-disjoint rainbow spanning forests to edge-disjiont rainbow spanning trees.
Let $G$ be an edge-colored multigraph. Let $F_1, \ldots, F_t$ be $t$ edge-disjoint
rainbow spanning forests. We are interested in whether $F_1,\ldots, F_t$ can be extended to
$t$ edge-disjoint rainbow spanning trees $T_1,\ldots, T_t$ in $G$, i.e., $E(F_i)\subset E(T_i)$ for each $i$.
We say the extension is {\em color-disjoint} if all edges in $\cup_i\lp E(T_i)\setminus E(F_i)\rp$
  have distinct colors and these colors are different from the colors appearing in the edges of $\cup_i E(F_i)$. Using similar matroid methods or graph theoretical arguments, we can also obtain a criterion that characterizes the existence of a color-disjoint extension of rainbow spanning forests into rainbow spanning trees.

  \begin{theorem}\label{extension}
    A family of $t$ edge-disjoint rainbow spanning forests $F_1, \ldots, F_t$ 
    has a color-disjoint extension in $G$ if and only if for every partition $P$ of $G$ into $|P|$ parts,
    \begin{equation}
      \label{eq:ext}
      |c(\CR(P,G'))|+\sum_{i=1}^t|\CR(P, F_i)|\geq t(|P|-1).
    \end{equation}
    Here $G'$ is the spanning subgraph of $G$ by removing all edges with colors appearing in some $F_i$,
    and $c(\CR(P,G'))$ be the set of colors appearing in the edges of $G'$ crossing the partition $P$.
  \end{theorem}

It would be interesting to find a similar criterion for the existence of $t$ edge-disjoint rainbow trees in a general graph since applications of Theorem \ref{partition} and Theorem \ref{extension} usually require large number of colors in the host graph.  

\vspace{0.2cm}

{\noindent}{\bf Organization:} The rest of the paper is organized as follows. In Section \ref{sec:color-disjoint}, we present the proofs of Theorem \ref{partition} and Theorem \ref{extension}. In Section \ref{sec:main}, we show Theorem \ref{anti-Ramsey}.

\section{Proof of Theorem \ref{partition}}\label{sec:color-disjoint}
We first reproduce the proof of Theorem \ref{partition} using matroid methods. A matroid is defined as $M = (E, \cI)$ where $E$ is the ground set and $\cI\subseteq 2^{E}$ is a set containing subsets of $E$ (called indepedent sets) that satisfy (i) if $A \subseteq B\subseteq E$, and $B\in \cI$, then $A\in \cI$; (ii) if $A\in \cI$, $B\in \cI$ and $|A| > |B|$, then $\exists\; a \in A\backslash B$ such that $B \cup \{a\} \in \cI$. Given a matroid $M = (E, \cI)$, the rank function $r_M: 2^E \to \mathbb{N}$ is defined as $r_M(S) = \max\{|I|: I\subseteq S, I\in \cI\}$. Thus $r_M(E)$ is the size of the maximum independent set of $M$. 
Two matroids of interests here are the \emph{graphic matroid} and the \emph{partition matroid}. Given an edge-colored graph $G$, the graphic matroid of $G$ is the matroid $M = (E, \cI)$ where $E = E(G)$ and $\cI$ is the set of forests in $G$. The partition matroid of $G$, is the matroid $M' = (E',\cI')$ where $E' = E(G)$ and $\cI$ is the set of rainbow subgraphs of $G$. Given $k$ matroids $\{M_i = (E_i, \cI_i)\}_{i\in[k]}$, one can define the \emph{union} of the $k$ matroids, $M_1 \vee \cdots \vee M_k = (E,\cI)$, by 
$E = \bigcup_{i=1}^k E_i$ and $\cI = \{I_1 \cup \cdots \cup I_k: I_i \in \cI_i \textrm{ for all $i \in [k]$}\}$. It is well known in matroid theory \cite{Edmondsu, Nash0} that $M_1 \vee \cdots \vee M_k$ is a matroid with rank function 
$$r(S) = \min_{T\subseteq S} \lp |S\backslash T| + \dss_{i=1}^k r_{M_i} (T\cap E_i) \rp.$$
Given two matroids $M_1 = (E, \cI_1)$ and $M_2 = (E, \cI_2)$ on the same ground set with rank functions $r_1$ and $r_2$ respectively, consider the family of independent sets common to both matroids, i.e., $\cI_1 \cap \cI_2$.  The well-known Matroid Intersection Theorem \cite{Edmondsi} asserts that
$$\max_{I\in \cI_1 \cap \cI_2} |I| = \min_{U\subseteq E} \lp r_1(U) + r_2(E\backslash U)\rp.$$

\subsection{Proof of Theorem \ref{partition} using Matroid methods}
Again we remark that the proof essentially follows the same approaches as Schrijver \cite{Schrijver} and we only reproduce it here for the sake of completeness. 

\begin{proof}[Proof of Theorem \ref{partition}]
The forward direction is clear. Thus it remains to show that if for every partition $P$ of $V(G)$ into $|P|$ parts, at least $t(|P|-1)$ distinct colors are represented in edges between partition classes, then there exist $t$ edge-disjoint rainbow spanning trees in $G$.

Given an edge-colored graph $G$, let $M = (E, \cI)$ be the graphic matroid of $G$ and $M' = (E, \cI')$ be the partition matroid of $G$. Moreover, let $M^t = M \vee M \vee \cdots \vee M = (E, \cI^t)$, where we take $t$ copies of $M$. By the matriod union theorem, we obtain that 
$$r_{M^t}(S) = \min_{T\subseteq S} \lp |S\backslash T| + t \cdot  r_{M}(T)\rp.$$
By the Matroid Intersection Theorem, 

\begin{align*}
    \max_{I \in \cI^t \cap \cI'} |I| &= \min_{U\subseteq E} \lp r_{M^t}(U) + r_{M'}(E\backslash U)\rp\\
                                     &= \min_{U\subseteq E} \lp \min_{T\subseteq U} \lp |U\backslash T| + t \cdot  r_{M}(T)\rp + r_{M'}(E\backslash U)\rp.
\end{align*}
Let $T, U\subseteq E$ be arbitrarily chosen such that $T\subseteq U$. 
Observe that 
$t \cdot r_M(T) = t(n-q(T))$,
where $q(T)$ is the number of components of $G[T]$.
Now we claim that $$|U\backslash T| + r_{M'}(E\backslash U) \geq r_{M'}(E\backslash T)\geq t(q(T)-1).$$ Indeed, for any color $c$ appearing in some edge  $e\in E\backslash T$, if $e\in E\backslash U$, then the color $c$ is counted in $r_{M'}(E\backslash U)$; if $e \in U$, then that color is counted in $|U\backslash T|$. 
In particular, at least $t(q(T)-1)$ distinct colors are represented in edges between connected components of $T$, thus in $E\setminus T$. 
It follows that 
$$|U\backslash T| + t \cdot  r_{M}(T) + r_{M'}(E\backslash U) \geq t(q(T)-1) + t(n-q(T)) \geq t(n-1),$$
which implies that $\max_{I \in \cI^t \cap \cI'} |I| \geq t(n-1)$.
   By definition, we then have $t$ edge-disjoint rainbow spanning trees.

\end{proof}

\subsection{Proof of Theorem \ref{partition} using graph theoretical arguments}
In this subsection, we give a new graph theoretical proof of Theorem \ref{partition}.
Given a graph $G$, we use $V(G),E(G)$ to denote its vertex set and edge set respectively. We use $\|G\|$ to denote the number of edges in $G$. Given a set of edges $E$, we use $c(E)$ to denote the set of colors that appear in $E$. For clarity, we abuse the notation to use $c(e)$ to denote the color of an edge $e$. We say a color $c$ has \textit{multiplicity} $k$ in $G$ if the number of edges with color $c$ in $G$ is $k$. The \textit{color multiplicity} of an edge in $G$ is the multiplicity of the color of the edge in $G$.

For any partition $P$ of the vertex set $V(G)$ and a subgraph $H$ of $G$, let $|P|$ denote the number of parts in the partition $P$ and let $\CR(P,H)$ denote the set of crossing edges in $H$ whose end vertices belong to different parts in the partition $P$. When $H=G$, we also write $\CR(P,G)$ as $\CR(P)$.
Given two partitions $P_1\colon V =\cup_i V_i$ and $P_2\colon V=\cup_j V'_j$, let the intersection $P_1\cap P_2$ denote the partition given by $V=\bigcup\limits_{i,j} V_i\cap V'_j$.  Given a spanning disconnected subgraph $H$, there is a natural partition $P_H$ associated to $H$, which partitions $V$ into its connected components. Without loss of generality, we abuse our notation $\CR(H)$ to denote
the crossing edges of $G$ corresponding to this partition $P_H$. 
Recall we want to show that an edge-colored multigraph $G$ has $t$ color-disjoint rainbow spanning trees if and only if for any partition $P$ of $V(G)$ (with $|P| \geq 2$),

\begin{equation}\label{eq:partition}
	|c(cr(P))| \geq t(|P|-1).
\end{equation}

\begin{proof}[Proof of Theorem \ref{partition}]
  One direction is easy. Suppose that $G$ contains $t$ pairwise color-disjoint rainbow  spanning trees $T_1, T_2, \ldots, T_t$. Then all edges in these trees have distinct colors. For any partition $P$ of the vertex set $V$, each tree contributes at least $|P|-1$ crossing edges, thus $t$ trees contribute  at least $t(|P|-1)$ crossing edges and the colors of these edges are all distinct.

  Now we prove the other direction. Assume that $G$ satisfies inequality
  \eqref{eq:partition}. We would like to prove $G$ contains $t$ pairwise color-disjoint rainbow  spanning trees. We will prove by contradiction.
  Assume that $G$ does not contain $t$ pairwise  color-disjoint rainbow spanning trees. Let ${\cal F}$ be the collection of all families of
  $t$ color-disjoint rainbow spanning forests $\{F_1, \cdots, F_t\}$. Consider the following deterministic process:

  \begin{center}
    \begin{tabbing}
      mmmm\=mmmm\=mmmm\=mmmm\=mmmm\= \kill
      Initially, set $C':=\bigcup\limits_{j=1}^t c(\CR(F_j))$\\
      {\bf while } $C'\not=\emptyset$ {\bf do}\\ 
      \>{\bf for } each color $x$ in $C' $, {\bf do}\\
      \>\> {\bf for } $j$ from 1 to $t$, {\bf do}\\
      \>\>\>{\bf if } color $x$ appears in $F_j$, {\bf then }\\
      \>\>\>\> delete the edge in color $x$ from $F_j$\\
      \>\>\>{\bf endif}\\
      \>\>{\bf endfor}\\
      \>{\bf endfor}\\
      \>set $C':=\bigcup\limits_{j=1}^t c(\CR(F_j))-C'$\\
      {\bf endwhile}\\
    \end{tabbing}
  \end{center}

For $i \geq 0$, $F_j^{(i)}$ denote the rainbow spanning forest $F_j$ after $i$ iterations of the while loop. In particular, $F_j^{(0)} = F_j$ for all $j\in [t]$ and $F_j^{(\infty)}$ is the resulting rainbow spanning forest of $F_j$ after the process. Similarly, let $C_i$ denote the set $C'$ after the $i$-th iteration of the while loop. Note that $C_i$ is the set of new colors crossing components of $F_j$s after some edges are deleted in the $i$-th iteration.

Observe that since the procedure is deterministic, $\{ F_j^{(i)}:  j\in [t], i>0\}$ is unique for a fixed family $\{F_1, \cdots, F_t\}$. We define a {\em preorder} on ${\cal F}$.
We say a family $\{F_j\}_{j=1}^t$ is less than or equal to
another family $\{F'_j\}_{j=1}^t$ if there is a positive integer $l$ such that
\begin{enumerate}
\item For $1\leq i<l$, $\dss_{j=1}^t \| F_j^{(i)}\|=\dss_{j=1}^t \| {F'}_j^{(i)}\|$.
\item $\dss_{j=1}^t \| F_j^{(l)}\| < \dss_{j=1}^t \| {F'}_j^{(l)}\|$.
\end{enumerate}

Since $G$ is finite, so is $\cal F$. There exists a maximal element
$\{F_1, F_2, \cdots, F_t\} \in {\cal F}$.  Run the deterministic process on $\{F_1, F_2, \cdots, F_t\}$.

The goal is to construct a common partition $P$ by refining
  $\CR(F_j)$ so that $|c(\CR(P))| < t(|P|-1)$. In particular, we will show that all forests  in $\{F^{(\infty)}_j: j\in [t]\}$ admit the same partition $P$.

  {\bf Claim (a):} $\bigcup\limits_{j=1}^t c \lp\CR(F^{(i)}_j)\rp \subseteq \lp \bigcup\limits_{j=1}^t c \lp\CR(F^{(i-1)}_j)\rp \rp\cup
  \lp\bigcup\limits_{j=1}^t c(F^{(i)}_j)\rp$.

  Assume for the sake of contradiction that there is a color $x\in \bigcup\limits_{j =1}^t c ( \CR(F_j^{(i)})) \backslash \bigcup\limits_{j =1}^t c(\CR(F^{(i-1)}_j))$ and there is no edge in color $x$ in all forests $F^{(i)}_1,\ldots, F^{(i)}_t$. Let $e$ be the edge such that $c(e)=x$ and $e\in \CR(F_s^{(i)})$ for some $s \in [t]$. Observe that since $c(e) \notin \bigcup\limits_{j =1}^t c(\CR(F^{(i-1)}_j))$, it follows that $F_s^{(i-1)}+ e$ contains a rainbow cycle, which passes through $e$ and another edge $e' \in F_s^{(i-1)}$ joining two distinct components of $F_s^{(i)}$. Now let us consider a new family of rainbow spanning forests $\{F_1', \cdots, F_t'\}$ where $F_j' = F_j$ for $j \neq s$ and $F_s' = F_s - e' + e$. The color-disjoint property is maintained since the color of edge $e$ is not in any $F_j$.
  Observe that since $c(e) \notin \bigcup\limits_{j =1}^t c(\CR(F^{(i-1)}_j))$, $F_s'^{(i)}$ will have one fewer component than $F_s^{(i)}$. Thus we have
\[\dss_{j=1}^t \| F_j^{(k)}\| = \dss_{j=1}^t \| F_j'^{(k)}\| \text{ for } k < i.\] 
\[\dss_{j=1}^t \| F_j'^{(i)}\| > \dss_{j=1}^t \| F_j^{(i)}\|.\] 
which contradicts our maximality assumption of $\{F_i: i\in [t]\}$. That finishes the proof of Claim $(a)$.

Claim (a) implies that for each $x \in C_i$, there is an edge $e$ of color $x$ in exactly one of the forests in $\{F_j^{(i)}: j\in [t]\}$. Thus removing that edge in the next iteration will increase the sum of number of partitions exactly by $1$. Thus we have that $$\dss_{j=1}^t|P_{F^{(i+1)}_j}|= \dss_{j=1}^t |P_{F^{(i)}_j}|+ |C_{i}|.$$

It then follows that
   \begin{align*}
     \sum_{j=1}^t|P_{F_j^{(\infty)}}| &= \sum_{j=1}^t |P_{F_j}|+ \sum_{i}|C_i|\\
     &= \sum_{j=1}^t |P_{F_j}|+ |\bigcup\limits_{j=1}^t c(\CR(F_j^{(\infty)}))|.
   \end{align*}
   
   Finally set the partition $P=\bigcap\limits_{j=1}^t P_{F^{(\infty)}_j}$.
   We claim $P_{F_j^{(\infty)}}=P$ for all $j$. This is because all edges in  $cr(P_{F_j^{(\infty)}}) \cap \bigcup\limits_{k=1}^t E(F_k^{(\infty)})$ have been already removed. We then have
   \begin{align*}
     t|P| &= \sum_{j=1}^t|P_{F^{(\infty)}_j}| \\
          &= \sum_{j=1}^t |P_{F_j}|+ |\bigcup\limits_{j=1}^t c(\CR(F^{(\infty)}_j))| \\
          &= \sum_{j=1}^t |P_{F_j}|+ |c(\CR(P))|\\
          &\geq t+1+ |c(\CR(P))|.
   \end{align*}
   We obtain
   $$ |c(\CR(P))| \leq t(|P|-1)-1.$$
   Contradiction.

\end{proof}   

\begin{corollary}
The edge-colored complete graph $K_n$ has $t$ color-disjoint rainbow spanning trees if the number of edges colored with any fixed color is at most $n/(2t)$.
\end{corollary}
\begin{proof}
	Suppose $K_n$ does not have $t$ color-disjoint rainbow spanning trees, then there exists a partition $P$ of $V(K_n)$ into $r$ parts ($2\leq r \leq n$) such that the number of distinct colors in the crossing edges of $P$ is at most $t(r-1)-1$. Let $m$ be the number of edges crossing the partition $P$. It follows that 
		$$m \leq \lp t(r-1)-1 \rp \cdot \frac{n}{2t}  \leq \frac{n}{2}(r-1) -\frac{n}{2t}. $$
	On the other hand, 
	$$m \geq \binom{n}{2} - \binom{n-(r-1)}{2}.$$
	Hence we have 
	$$\binom{n}{2} - \binom{n-(r-1)}{2} \leq \frac{n}{2}(r-1) -\frac{n}{2t}.$$ which implies
	$$(n-r)(r-1) \leq -\frac{n}{t}.$$ which contradicts that $2\leq r\leq n$.
\end{proof}

{\bf Remark:} This result is tight since the total number of colors used in $K_n$ could be
as small as ${n\choose 2}/(n/(2t))= t(n-1),$  but
any $t$ color-disjoint rainbow spanning trees need $t(n-1)$ colors.
On the contrast, a result by Carraher, Hartke and Horn \cite{CHH} implies
there are $\Omega(n/\log n)$ edge-disjoint rainbow spanning trees.


\subsection{Proof of Theorem \ref{extension}}
  
  Recall we want to show that any $t$ edge-disjoint rainbow spanning forests $F_1, \ldots, F_t$ have a color-disjoint extension to edge-disjoint rainbow spanning trees in $G$ if and only if

  $$ |c(\CR(P,G'))|+\sum_{j=1}^t|\CR(P, F_j)|\geq t(|P|-1).$$
  where $G'$ is the spanning subgraph of $G$ by removing all edges with colors appearing in some $F_j$.

  \begin{proof}
 	Again, the forward direction is trivial. We only need to show that condition \eqref{eq:ext} implies there exists a color-disjoint extension to edge-disjoint rainbow spanning trees.
    The proof is similar to the proof of Theorem \ref{partition}. Consider a set of edge-maximal forests $F^{(0)}_1,\ldots, F^{(0)}_t$ which is a color-disjoint extension of $F_1, \ldots, F_t$.
    From $\{F^{(0)}_j\}$ we delete
    all edges (in $\{F^{(0)}_j\}$) of some color $c$ appearing in $\bigcup_{j=1}^tc(\CR(F^{(0)}_j,G'))$ to get
      a new set $\{F^{(1)}_j\}$. Repeat this process until we reach a stable set $\{F^{(\infty)}_j\}$.
      Since we only delete edges in $G'$, we have $E(F_j)\subseteq E(F^{(\infty)}_j)$ for each $1\leq j \leq t$.
      The edges and colors in $\cup_{j=1}^t E(F_j)$ will not affect the process. A similar claim still holds:

      $$\bigcup\limits_{j=1}^t c(\CR(F^{(i)}_j, G')) \subseteq  \lp \bigcup\limits_{j=1}^t c(\CR(F^{(i-1)}_j,G')) \rp \cup
  \lp\bigcup\limits_{j=1}^t c \lp E(F^{(i)}_j)\cap E(G')\rp\rp.$$

  In particular, let $C_i = \lp\bigcup_{j=1}^t c(\CR(F^{(i)}_j,G')) \rp \backslash \lp \bigcup_{j=1}^t c(\CR(F^{(i-1)}_j,G')) \rp $. Then we have
    $$\dss_{j=1}^t|P_{F^{(i+1)}_j}| = \dss_{j=1}^t |P_{F^{(i)}_j}|+ |C_{i}|.$$ 
It then follows that
   \begin{align*}
     \sum_{j=1}^t|P_{F_j^{(\infty)}}| &= \sum_{j=1}^t |P_{F_j^{(0)}}|+ \sum_{i}|C_i|\\
     &= \sum_{j=1}^t |P_{F_j^{(0)}}|+ |\bigcup\limits_{j=1}^t c(\CR(F_j^{(\infty)},G'))|.
   \end{align*}
   
    Finally set the partition $P=\bigcap\limits_{j=1}^t P_{F^{(\infty)}_j \backslash E(F_j)}$.
  Clearly all edges in $\CR(P, G')$ are removed.
  All possible edges remaining in $G$ that cross the partition $P$ are exactly the edges in $\bigcup\limits_{j=1}^t\CR(P, F_j)$.

   We have
   \begin{align*}
     t|P| &= \sum_{j=1}^t|P_{F^{(\infty)}_j}| + \sum_{j=1}^t|\CR(P, F_j)| \\
          &= \sum_{j=1}^t |P_{F_j^{(0)}}|+ |\bigcup\limits_{j=1}^t c(\CR(F^{(\infty)}_j,G'))| + \sum_{j=1}^t|\CR(P, F_j)|\\
          &= \sum_{j=1}^t |P_{F_j^{(0)}}|+ |c(\CR(P,G'))|+  \sum_{j=1}^t|\CR(P, F_j)|\\
          &\geq t+1+ |c(\CR(P,G'))|+ \sum_{j=1}^t|\CR(P, F_j)|.
   \end{align*}
   We obtain
   $$|c(\CR(P,G'))|+ \sum_{j=1}^t|\CR(P, F_j)|\leq t(|P|-1)-1.$$
Contradiction.

  \end{proof}

\section{Proof of Theorem \ref{anti-Ramsey}}\label{sec:main}

Recall that $r(n,t)$ is the maximum number of colors in
an edge-coloring of the complete graph $K_n$ not having $t$ edge-disjoint rainbow spanning trees.\\

\noindent{\bf Lower Bound:} Jahanbekam and West (See Lemma 5.1 in \cite{West}) showed the following lower bound for $r(n,t)$. 
\begin{proposition}{\cite{West}}\label{lower-bound}
For positive integers $n$ and $t$ such that $t\leq 2n-3$, there is an edge-coloring of $K_n$ using $\binom{n-2}{2} + t$ colors that does not have $t$ edge-disjoint rainbow spanning trees. When $n=2t+1$, the construction improves to $\binom{n-1}{2}$ colors. When $n=2t$, it improves to $\binom{n}{2}-t$. 
\end{proposition}

This matches the upper bounds in Theorem \ref{anti-Ramsey}. Hence we will skip the proof of lower bounds in the subsequent theorems. Moreover, we only consider the case $t\geq 2$ since the case $t=1$ was already resolved in Bialostocki and Voxman \cite{BV}. In Section \ref{sec:tech}, we prove a technical lemma that will be used in the proof of Theorem \ref{anti-Ramsey}. In Section \ref{sec:2t+2}, \ref{sec:2t+3},\ref{sec:2t+1}, we show Theorem \ref{anti-Ramsey} when $n$ is in different range of values with respect to $t$.

\subsection{Technical lemma}\label{sec:tech}

\begin{lemma}\label{lem:1}
 Let $G$ be an edge-colored graph with $s$ colors $c_1, \cdots, c_s$ and $|V(G)| =n = 2t+2$ where $t\geq 3$. For color $c_i$, let $m_i$ be the number of edges of color $c_i$. Suppose $\dss_{i=1}^s (m_i-1) = 3t$ and $m_i \geq 2$ for all $i \in [s]$. Then we can construct $t$ edge-disjoint rainbow forests $F_1,\ldots, F_t$ in $G$ such that if we define $G_0 = G  -\bigcup\limits_{i=1}^t E(F_i)$, then 
  \begin{equation} \label{eq:5}
  |E(G_0)| \leq 2t+1.
  \end{equation} and 
 \begin{equation} \label{eq:6}
 \Delta(G_0) \leq t+1.
  \end{equation}
 
\end{lemma}

\begin{proof}
  We consider two cases:
  \begin{description}
  \item Case 1: $m_1\geq 2t+2$. 
    Note that
    $$\sum_{i=2}^s(m_i-1)=3t-(m_1-1)\leq t-1.$$
    Thus, $s\leq t$.  Let $d_i(v)$ be the number of edges
    in color $c_i$ and incident to $v$ in the current graph $G$.  We
    construct the edge-disjoint rainbow forests $F_1,F_2,\ldots, F_t$
    in two rounds: In the first round, we greedily extract edges only in
    color $c_1$. For $i=1,\ldots, t$, at step $i$, pick a vertex $v$ with
    maximum $d_1(v)$ (pick arbitrarily if tie).
    Pick an edge in color $c_1$ incident to $v$,
    assign it to $F_i$, and delete it from $G$.
   
    We claim that after
    the first round $d_1(v)\leq t+1$ for any vertex $v$.

    Suppose not,
    if $d_1(v)\geq t+2$. Since $n-1-(t+2) < t$, it follows that there exists another vertex $u$ with $d_1(u)\geq d_1(v)-1\geq t+1$.
    
    This implies
    $$m_1\geq t+d_1(v)+d_1(u)-1\geq 3t+2.$$
    However, $$m_1-1\leq \sum_{i=1}^s(m_i-1)=3t.$$
    which gives us the contradiction. 

    In the second round, we greedily extract edges not in color $c_1$.
    For $i=1,\ldots, t$, at step $i$, among all vertices $v$ with at least
    one neighboring edge not in color $c_1$,
    pick a vertex $v$ with
    maximum vertex degree $d(v)$ (pick arbitrarily if tie).
    Pick an edge  incident to $v$ and not in color $c_1$,
    assign it to $F_i$, and delete it from $G$. 
        
    If we succeed with selecting $t$ edges not in color $c_1$ in the second round,
    we claim $d(v)\leq t+1$ for any vertex $v$. 
    Suppose not,
    if $d(v)\geq t+2$. Then there is another vertex $u$ with $d(u)\geq d(v)-1\geq t+1$.
    It implies
    $$\sum_{i=1}^sm_i\geq 2t+d(u)+d(v)-1\geq 4t+2.$$
    However, since $s\leq t$, we have
    $$\sum_{i=1}^s m_i\leq 3t+s\leq 4t.$$
    Contradiction. Therefore it follows that $d(v) \leq t+1$. Moreover, $|E(G_0)| \leq 4t - 2t \leq 2t$.

    If the process stops at step $i=l<t$, then all remaining edges in $G_0$ must be in color 1.
    Thus, by the previous claim, $\Delta(G_0) \leq t+1$. Moreover,
    $$|E(G_0)| \leq m_1 -t \leq (3t+1) -t = 2t+1.$$

	In both cases above, $F_1, \cdots F_t$ are edge-disjoint rainbow forests that satisfies inequality (\ref{eq:5}) and (\ref{eq:6}).

  \item Case 2: $m_1\leq 2t+1$. 
  
  {\bf Claim:} There exists $t$ edge-disjoint rainbow forests $F_1, F_2, \cdots, F_t$ such that $\Delta(G_0) \leq t+1$. 
  
  For $j=1,2,\ldots, t$, we will construct a rainbow forest $F_j$ by selecting a rainbow set of edges such that after deleting these edges from $G$, $\Delta(G_0) \leq 2t+1-j$. Notice that when $j=t$, we will have $\Delta(G_0) \leq t+1$. Our procedure is as follows:
  
  For step $j$, without loss of generality, let $v_1, v_2, \cdots, v_l$ be the vertices with degree $2t+2 -j$ and let $c_1, c_2, \cdots, c_m$ be the set of colors of edges incident to $v_1, v_2, \cdots, v_l$ in $G$. If there is no such vertex, simply pick an edge incident to the max-degree vertex and assign it to $F_j$. Otherwise, we will construct an auxiliary bipartite graph $H = A\cup B$ where $A = \{v_1, \cdots, v_l\}$ and B = $\{c_1, c_2, \dots, c_m\}$ and $v_x c_y \in E(H)$ if and only if there is an edge of color $c_y$ incident to $v_x$. We claim that there exists a matching of $A$ in $H$. Suppose not, then by Hall's theorem, there exists a set of vertices $A' = \{u_1, u_2, \cdots u_k\} \subseteq A$ such that $|N(A')|< |A'|=k$ where $k\geq 2$. Without loss of generality, suppose $N(A) = \{c_1', c_2', \cdots, c_q'\}$ where $q\leq k-1$. Let $m_i'$ be the number of edges of color $c_i'$ remaining in $G$.
  
  Note that $k\neq 2$ since otherwise we will have one color with at least $2\cdot (2t+2-j) -1 \geq 2t+3$ edges, which contradicts our assumption in this case.
  
  Notice that for every $i\in [k]$, $u_i$ has at least $(2t+2-j)$ edges incident to it. Moreover, at least $j-1$ edges are already deleted from $G$ in previous steps. Therefore, we have 
  \begin{align*}
	  \frac{k(2t+2-j)}{2} & \leq \dss_{i=1}^{q} m_i' \leq \lp \dss_{i=1}^{q} (m_i'-1)\rp + (k-1) \leq 3t - (j-1) + (k-1).
  \end{align*}
It follows that $$k \leq 2+ \frac{2t}{2t-j} \leq 4.$$
Similarly, using another way of counting the edges incident to some $u_i$ $(i\in [k]$), we have 
$$ k(2t+2-j) - \binom{k}{2} \leq 3t-(j-1)+(k-1).$$
which implies that 
$$t(2k-3) \leq \frac{k(k-3)}{2} + j(k-1) \leq \frac{k(k-3)}{2} + t(k-1).$$

 It follows that $t\leq \frac{k(k-3)}{2(k-2)}.$ Since $k \leq 4$ and $k> 2$, we obtain that $t\leq 1$, which contradicts our assumption that $t\geq 2$. Thus by contradiction, there exists a matching of $A$ in $H$. This implies that there exists a rainbow set of edges $E_j$ that cover all vertices with degree $2t+2 -j$ in step $j$. We can then find a maximally acyclic subset $F_j$ of $E_j$ such that $F_j$ is a rainbow forest and every vertex of degree $2t+2-j$ is adjacent to some edge in $F_j$. Delete edges of $F_j$ from $G$ and we have $\Delta(G_0) \leq 2t+1 -j$. As a result, after $t$ steps, we obtain $t$ edge-disjoint rainbow forests $F_1, \cdots, F_t$ and $\Delta(G_0) \leq t+1$. This finishes the proof of the claim.\\

Now let $\{F_1, F_2, \cdots, F_t\}$ be an edge-maximal set of $t$ edge-disjoint rainbow forests that satisfies $\Delta(G_0) \leq t+1$. We claim that $|E(G_0)| \leq 2t+1$. Suppose not, i.e., $|E(G_0)| \geq 2t+2$. It follows that $\dss_{i=1}^t |E(F_i)| \leq 6t-(2t+2) < 4t$, i.e. there exists a $j \in [t]$ such that $F_j$ has at most $3$ edges. Since $F_j$ is edge maximal, none of the edges in $G_0 $ can be added to $F_j$.  We have three cases:

 \begin{description}
 	\item Case 2a: $|E(F_j)| = 1$. It then follows that all edges in $G_0$ have the same color (call it $c_1'$) as the single edge in $F_j$. Thus we have a color with multiplicity at least $2t+3$, which contradicts that $m_1 < 2t+2$.

 	\item Case 2b: $|E(F_j)| = 2$. Similarly, we have that at least $2t+1$ edges in $G_0$ that share the same colors (call them $c_1', c_2'$) as edges in $F_j$. It follows that $m_1 + m_2 \geq 2t+3$. Similar to Case 1, in this case, we have that $s \leq t+1$ and $|E(G)| = 3t+s \leq 4t+1$. Since $|E(G_0)| \geq 2t+2$, that implies that $\dss_{i=1}^t |E(F_i)| \leq (4t+1)-(2t+2) = 2t-1.$ Hence there exists some $F_k$ such that $|E(F_k)| \leq 1$ and we are done by Case 2a.

 	\item Case 2c: $|E(F_j)| = 3$. Similarly, we have that at least $2t-1$ edges in $G_0$ share the same colors (call them $c_1', c_2', c_3')$ as edges in $F_j$. It follows that  $m_1 + m_2 + m_3 \geq 2t+2$. By inequality \eqref{edge-bound}, we have that $s \leq t+4$ and $|E(G)| \leq 4t+4$. Since $|E(G_0)| \geq 2t+2$, that implies that $\dss_{i=1}^t |E(F_i)| \leq 2t+2$. Since $t\geq 3$ by our assumption, there exists a $k \in [t]$ such that $|E(F_k)| \leq 2$ and we are done by Case $2b$ and Case $2c$.
 	
 		
 \end{description}

 Therefore, by contradiction, we have that $|E(G_0)| \leq 2t+1$ and we are done.

  \end{description}
  
\end{proof}


\subsection{Proof of Theorem \ref{anti-Ramsey} where $n =2t+2$}\label{sec:2t+2}

\begin{proposition}\label{base-case}
	For any $n = 2t+2 \geq 6$, we have $r(n,t) = \binom{n-2}{2} +t = 2t^2$.
\end{proposition}

\begin{proof}
Note that the lower bound is shown by Jahanbekam and West in Proposition \ref{lower-bound}. For the upper bound, we will assume that $t\geq 3$ since the case when $t = 2$ is implied by the result of Akbari and Alipour \cite{AA}. We will show that any coloring of $K_{2t+2}$ with $2t^2+1$ distinct colors contains $t$ edge-disjoint rainbow spanning trees. Call this edge-colored graph $G$. Let $m_i$ be the multiplicity of the color $c_i$ in $G$. Without loss of generality, say the first $s$ colors have multiplicity at least $2$, i.e.
\[m_1 \geq m_2 \geq \cdots \geq m_s \geq 2.\]

Let $G_1$ be the spanning subgraph of $G$ consisting of all edges with color multiplicity greater than 1 in $G$. Let $G_2$ be the spanning subgraph consisting of the remaining edges.
  We have
  \begin{equation}\label{edge-bound}
  \sum_{i=1}^s (m_i-1)={n\choose 2}-(2t^2 + 1)=3t.
\end{equation}

  In particular, we have
  $$|E(G_1)|=\sum_{i=1}^s m_i=3t+s\leq 6t.$$
 
By Lemma \ref{lem:1}, it follows that we can construct $t$ edge-disjoint rainbow spanning forests $F_1,\ldots, F_t$ in $G$ such that if we define $G_0 = E(G_1)  -\bigcup\limits_{i=1}^t E(F_i)$, then 
  $$ |E(G_0)| \leq 2t+1.$$ and $$\Delta(G_0) \leq t+1.$$

Now we show that  $F_1,\ldots, F_t$ have a color-disjoint extension to $t$ edge-disjoint rainbow spanning trees.
Consider any partition $P$. We will verify
\begin{equation} \label{eq:extend}
|c(\CR(P),G_2)|+ \sum_{i=1}^t|\CR(P, F_i)|\geq t(|P|-1).
\end{equation}

We will first verify the case when $3\leq |P| \leq n$. Note that
$$|c(\CR(P),G_2)|+ \sum_{i=1}^t|\CR(P, F_i)| -t(|P|-1) \geq {n\choose 2}-(2t+1) -{n-|P|+1 \choose 2}-t(|P|-1).$$

We want to show that the right hand side of the above inequality is nonnegative.
  Note that the function on the right hand side is concave downward with respect to $|P|$. 
  Thus it is sufficient to verify it at $|P|=3$ and $|P|=n$.

When $|P|=3$, we have
  $${n\choose 2}-(2t+1) -{n-2\choose 2}-2t = 0.$$

When $|P|=n$, we have
$${n\choose 2}-(2t+1)-t(n-1)=0.$$

It remains to verify the inequality (\ref{eq:extend}) for $|P| = 2$.
By Theorem \ref{extension}, we have $|E(G_0)|\leq 2t+1$.
If each part of $P$ contains at least $2$ vertices, then we have
\begin{align*}
  &\hspace*{-1cm} |c(\CR(P),G_2)|+ \sum_{i=1}^t|\CR(P, F_i)| -t(|P|-1)\\
  &\geq {n\choose 2}-|E(G_0)|-\lp{n-2 \choose 2}+1 \rp-t\\
  &\geq {n\choose 2}-(2t+1) -\lp{n-2 \choose 2}+1 \rp-t\\
  &=t-1\geq 0.
\end{align*}
Otherwise, $P$ is of the  form $V(G) = \{v\} \cup B$ for some $v\in V(G)$ and $B = V(G)\backslash \{v\}$.
By Lemma \ref{lem:1}, we have $d_{G_0}\leq t+1$. Thus,
$$|c(\CR(P),G_2)|+ \sum_{i=1}^t|\CR(P, F_i)| - t(|P|-1)\geq (n-1)-d_{G_0}(v)-t
\geq 2t+1-(t+1)-t= 0.$$
Therefore, by Theorem \ref{extension}, $F_1,\ldots, F_t$ have a color-disjoint extension to $t$ edge-disjoint rainbow spanning trees.
\end{proof}


\subsection{Proof of Theorem \ref{anti-Ramsey} where $n\geq 2t+3$}\label{sec:2t+3}

\begin{proposition}
	For any $n\geq 2t+2 \geq 6$, we have  $r(n,t)={n-2\choose 2}+t$.
\end{proposition}

\begin{proof}
	Again, the lower bound is due to Proposition \ref{lower-bound}. For the upper bound, we will show that every edge-coloring of $K_n$ with exactly ${n-2\choose 2} +t+1$
  distinct colors has $t$ edge-disjoint spanning trees. Call this edge-colored graph $G$.

Given a vertex $v$, we define $D(v)$ to be the set of colors $C$ such that every edge with colors in $C$ is incident to $v$. Given a vertex $v$ and a set of colors $C$, define $\Gamma(v, C)$ as the set of edges incident to $v$ with colors in $C$. For ease of notation, we let $\Gamma(v) = \Gamma(v,D(v))$.

For fixed $t$, we will prove the theorem by induction on $n$.
The base case is when $n = 2t+2$, which is proven in Proposition \ref{base-case}. Let's now consider the theorem when $n\geq 2t+3$.

\begin{description}
\item Case 1: there exists a vertex $v \in V(G)$ with $|\Gamma(v)| \geq t$ and $|D(v)| \leq n-3$.

  In this case, we set $G' = G-\{v\}$. Note that $G'$ is an
  edge-colored complete graph with at least
  ${n-2\choose 2} +t+1 - (n-3) = {n-3 \choose 2} + t+1$ distinct
  colors.  Moreover $|G'| \geq 2t+2$. Hence by induction, there exists
  $t$ edge-disjoint rainbow spanning trees in $G'$.  Note that by our
  definition of $D(v)$, none of the colors in $D(v)$ appear in
  $E(G')$. Moreover, since $|\Gamma(v)| \geq t$, we can extend the $t$
  edge-disjoint rainbow spanning trees in $G'$ to $G$ by adding one
  edge in $\Gamma(v)$ to each of the rainbow spanning trees in $G'$.

\item Case 2: Suppose we are not in Case 1. We first claim that
 there exists two vertices $v_1, v_2 \in V(G)$ such that $|\Gamma(v_1)| \leq t-1$ and $|\Gamma(v_2)| \leq t-1$.

  Otherwise, there are at least $n-1$ vertices $u$ with $|\Gamma(u)| \geq t$. Since we are not in Case $1$, it follows that all these vertices $u$ also satisfy $|D(u)| \geq n-2$.		
  Hence by counting the number of distinct colors in $G$, we have that 
  \[   \frac{(n-1)(n-2)}{2} \leq {n-2 \choose 2} + t + 1.\]
  which implies that $n \leq t+3$, giving us the contradiction.\\

 Now suppose $|\Gamma(v_1)| \leq t-1$ and $|\Gamma(v_2)| \leq t-1$. Let $D = D(v_1) \cup D(v_2)$. Add new colors to $D$ until $|\Gamma(v_1, D)| \geq t$, $|\Gamma(v_2, D)| \geq t+1$ and $|D| \geq t+1$. Call the resulting color set $S$. Note that 
					\[ t+1 \leq |S| \leq 2t+1 \leq n-2.\]
		
		Now let $G' = G-\{v_1, v_2\}$ and delete all edges of colors in $S$ from $G'$. 
		
		We claim that $G'$ has t color-disjoint rainbow spanning trees.

		By Theorem \ref{partition}, it is sufficient to verify the condition that for any partition $P$ of $V(G')$, 
     					 $$|c(\CR(P, G'))|\geq t(|P|-1).$$
 			 Observe    
			       \begin{align*}
			        & |c(\CR(P,G'))| -t(|P|-1) \hspace*{-2cm} \\
			        & \geq 
			        |c(E(G')|- {n-1-|P|\choose 2}-t(|P|-1)\\
			                   &\geq {n-2\choose 2}+t+1 -|S| - {n-1-|P|\choose 2}-t(|P|-1)\\
			                    &\geq {n-2\choose 2}+t+1 - (n-2) - {n-1-|P|\choose 2}-t(|P|-1).
			      \end{align*}
			      Note the expression above is concave downward as a function of $|P|$. It is sufficient to check the value
			      at $2$ and $n-2$.
			      When $|P|=2$, we have
			      \begin{align*}
			 |c(\CR(P,G'))| -t(|P|-1) \geq 
			{n-2\choose 2}+t+1 - (n-2) - {n-3\choose 2}-t= 0.       
			      \end{align*}
			
			When $|P|=n-2$, we have
			\begin{align*}
			 |c(\CR(P,G'))| -t(|P|-1)&\geq {n-2\choose 2}+t+1 - (n-2) -t(n-3)\\
			&= \frac{(n-4)(n-2t-3)}{2}\\
			&\geq 0.
			\end{align*}
			Here we use the assumption $n \geq 2t+3$ in the last step.
            Now it remains to extend the $t$ color-disjoint spanning trees we found to $G$ by using only the colors in $S$.
			Let $e_1, \cdots, e_k$ be the edges in $G$ incident to $v_1$ with colors in $S$. Let $e_1', \cdots e_l'$ be the edges in $G\backslash \{v_1\}$ incident to $v_2$ with colors in $S$. With our selection of $S$, it follows that $k, l\geq t$. Now construct  an auxiliary bipartite graph $H$ with partite sets $A = \{e_1, \cdots, e_k\}$ and $B = \{e_1', \cdots, e_l'\}$ such that $e_i e_j' \in E(H)$ if and only if $e_i, e_j'$ have different colors in $G$.  

			We claim that there is a matching of size $t$ in $H$.
			Let $M$ be the maximum matching in $H$. Without loss of generality, suppose $e_1 e_1', \cdots, e_m e_m' \in M$ where $m < t$. It follows that $\{e_j:  m < j \leq k\} \cup \{e_j':  m < j \leq l\}$ all have the same color (otherwise we can extend the matching). Without loss of generality, they all have color $x$. Now observe that for every matched edge $e_i e_i'$, exactly one of the two end vertices must be in color $x$. Otherwise, we can extend the matching by pairing $e_i$ with $e_t'$ and $e_t$ with $e_i'$. This implies that $H$ has at most $t$ colors, which contradicts that $|S| \geq t+1$.

			Hence there is a matching of size $t$ in $H$. Since none of the edges in $G'$ have colors in $S$, it follows that we can extend the $t$ color-disjoint rainbow spanning trees in $G'$ to $t$ edge-disjoint rainbow spanning trees in $G$.

\end{description}

Hence in all of the three cases, we obtain that $G$ has $t$ edge-disjoint rainbow spanning trees.

\end{proof}

\subsection{Theorem \ref{anti-Ramsey} where $n=2t+1$}\label{sec:2t+1}

\begin{proposition}
	For positive integers $t \geq 1$ and $n=2t+1$, we have $r(n,t)={n-1\choose 2}=2t^2-t$.
\end{proposition}

\begin{proof}
	Again, the lower bound is due to Proposition \ref{lower-bound}.
	Now we prove that any edge-coloring of $K_{2t+1}$ with $2t^2-t+1$ distinct colors contains $t$ edge-disjoint rainbow spanning trees. 
	Call this edge-colored graph $G$. The proof approach is similar to the case when $n=2t+2$.
	Let $m_i$ be the multiplicity of the color $c_i$ in $G$.
	Without loss of generality,  say the first $s$ colors have multiplicity greater than or equal to $2$:
	$$m_1\geq m_2\geq \cdots \geq m_s\geq 2.$$
	Let $G_1$ be the spanning subgraph consisting of all edges whose color multiplicity is greater than 1 in $G$. Let $G_2$ be the spanning subgraph consisting of the remaining edges.
	We have
	\begin{equation}\label{eq:m2}
	\sum_{i=1}^s (m_i-1)={n\choose 2}-(2t^2-t+1)=2t-1.
	\end{equation}
	
	In particular, we have
	$$|E(G_1)|=\sum_{i=1}^s m_i=2t-1+s\leq 4t-2.$$
	\textbf{Claim:} we can construct $t$ edge-disjoint rainbow forests $F_1,\ldots, F_t$ in $G_1$ such that if we let $G_0 = G_1 \backslash \bigcup\limits_{i=1}^t E(F_i)$, then $|E(G_0)| \leq t$. Again, for the proof of the claim, we consider two cases:
	
	\begin{description}
		
		\item Case 1: $m_1 \geq t+2$. By equation (\ref{eq:m2}), we have that $s \leq (2t-1)-(t+1)+1 = t-1.$ We construct $t$ edge-disjoint rainbow forests $F_1, \cdots, F_t$ as follows: First take $t$ edges of color $c_1$ and add one edge to each of $F_1,\cdots F_t$. Next, pick one edge from each of the remaining $s-1$ colors and add each of them to a distinct $F_i$.
		
		Clearly, we can obtain $t$ edge-disjoint rainbow forests in this way. Furthermore, $$|E(G_0)| \leq 2t-1 + s - (t+s-1) = t.$$ which proves the claim.
		
		\item Case 2: $m_1 < t+2$. 
		Let $F_1,\ldots, F_t$ be the edge-maximal family of rainbow spanning forests in $G_1$. Let $G_0 = G_1 \backslash \bigcup\limits_{i=1}^t E(F_i)$. Suppose that $|E(G_0)| > t$. Then $$\sum_{i=1}^t |E(F_i)| \leq 2t-1+s-(t+1) = t+s-2.$$
		
		Since $s\leq 2t-1$, it follows that there exists some $j$ such that $|E(F_j)| \leq 2$. 
		\begin{description}
			\item Case 2a: $|E(F_j)| = 1$. Since $\{F_1,\ldots, F_t\}$ is edge-maximal and $|E(G_0)|\geq t+1$, it follows that all edges in $G_0$ share the same color (call it $c_1'$) as the single edge in $F_j$. Thus $m_1 \geq t+2$, which contradicts that $m_1 < t+2$ since we are in Case 2.

			\item Case 2b: $|E(F_j)| = 2$. Similarly, at least $t$ edges in $G_0$ share the same colors (call them $c_1'$, $c_2'$) as the two edges in $F_j$.	It follows that $m_1 + m_2 \geq t+2$. Hence 
			$s\leq t+1$.
			
			Now since $|E(G_0)|\geq t+1$, it follows that $$\sum_{i=1}^t |E(F_i)| \leq 2t-1+s -(t+1) = t+s-2 \leq 2t-1,$$ Hence there exists some forest with only one edge, in which case we are done by Case 2a.
			
		\end{description}
		Hence by contradiction, we obtain that $|E(G_0)| \leq t$, which completes the proof of the claim.\\
		
	\end{description}

	Now we show that  $F_1,\ldots, F_t$ have a color-disjoint extension to $t$ edge-disjoint rainbow spanning trees.
	Consider any partition $P$. We will verify
	$$|c(\CR(P),G_2)|+ \sum_{i=1}^t|\CR(P, F_i)|\geq t(|P|-1).$$
	We have
	$$|c(\CR(P),G_2)|+ \sum_{i=1}^t|\CR(P, F_i)| -t(|P|-1)
	\geq {n\choose 2}-t -{n-|P|+1 \choose 2}-t(|P|-1).$$
	Note that the function on right is concave downward on $|P|$.
	It is enough to verify it at $|P|=2$ an $|P|=n$.
	When $|P|=2$, we have
	$${n\choose 2}-t -{n-1\choose 2}-t =n-1-2t\geq 0.$$
	When $|P|=n$, we have
	$${n\choose 2}-t-t(n-1)=0.$$
	By Theorem \ref{extension}, $F_1,\ldots, F_t$ have a color-disjoint extension to $t$ edge-disjoint rainbow spanning trees.
\end{proof}

{\bf Acknowledgement} The authors thank an anonymous referee
for the valuable comments, in particular, for pointing out that Schrijver's theorem implies Theorem 3.



\end{document}